\documentclass[english]{amsart}
\usepackage[T1]{fontenc}
\usepackage[latin9]{inputenc}
\usepackage{mathrsfs}
\usepackage{amstext}
\usepackage{amsthm}
\usepackage{amssymb}

\makeatletter
\numberwithin{equation}{section}
\numberwithin{figure}{section}
\theoremstyle{plain}
\newtheorem{thm}{\protect\theoremname}
\theoremstyle{definition}
\newtheorem{defn}[thm]{\protect\definitionname}
\theoremstyle{plain}
\newtheorem{lem}[thm]{\protect\lemmaname}
\theoremstyle{plain}
\newtheorem{fact}[thm]{\protect\factname}

\makeatother

\usepackage{babel}
\providecommand{\definitionname}{Definition}
\providecommand{\factname}{Fact}
\providecommand{\lemmaname}{Lemma}
\providecommand{\theoremname}{Theorem}

\begin{document}

\title[$IP^{\star}$ set in product space of partial semigroups]{{\large{}$IP^{\star}$ set in product space of countable adequate
commutative partial semigroups}}

\author{Aninda Chakraborty }

\address{{\large{}Aninda Chakraborty, Department of mathematics, Government
General Degree College Chapra, University of Kalyani}}

\email{{\large{}anindachakraborty2@gmail.com}}

\begin{abstract}
{\large{}A partial semigroup is a set with restricted binary operation.
In this work we will extend a result due to V. Bergelson and N. Hindman
concerning the rich structure presented in the product space of semigroups
to partial semigroup. An $IP^{\star}$ set in a semigroup is a set
that intersect every set of the form $\left\{ FS(x_{n})_{n=1}^{\infty}:x_{n}\in S\right\} $.
V. Bergelson and N. Hindman proved that if $S_{1},S_{2},\ldots,S_{l}$
are finite collection of commutative semigroup, then under certain
condition, an $IP^{\star}$ set in $S_{1}\times S_{2}\times\ldots\times S_{l}$
contains cartesian products of arbitrarily large finite substructures
of the form $FS\left(x_{1,n}\right)_{n=1}^{\infty}\times FS\left(x_{2,n}\right)_{n=1}^{\infty}\times\ldots\times FS\left(x_{l,n}\right)_{n=1}^{\infty}$.
In this work we will extend this result to countable adequate commutative
partial semigroup.}{\large\par}
\end{abstract}

\maketitle

\section{{\large{}introduction}}

{\large{}A partial semigroup is a set with restricted binary operation
defined as,}{\large\par}
\begin{defn}
{\large{}\label{Defn 1} A partial semigroup is defined as a pair
$(S,\ast)$ where $\ast$ maps a subset of $S\times S$ to $S$ and
satisfies for all $a,b,c\in S$ , $(a\ast b)\ast c=a\ast(b\ast c)$
in the sense that if either side is defined, then so is the other
and they are equal.}{\large\par}
\end{defn}

{\large{}There are several notion of largeness of sets in partial
semigroup similar to semigroup arises from similar characterization
of the structure of Stone-\v{C}ech compactification of semigroup,
$\beta S$. The Stone-\v{C}ech compactification of a partial semigroup
contains an interesting subsemigroup, which gives us many important
notions of largeness of sets. }{\large\par}

{\large{}For a commutative semigroup $(S,+)$, the Stone-\v{C}ech
compactification, $\beta S$, is the collection of all ultrafilters
on $S$. Identifying the principal ultrafilters with the points of
$S$ we see that $S\subseteq\beta S$. Given$A\subseteq S$ let us
set, 
\[
\overline{A}=\{p\in\beta S:\:A\in p\}.
\]
 Then the set $\{\overline{A}:\:A\subseteq S\}$ is a basis for a
topology on $\beta S$. The operation $'+'$ on $S$ can be extended
to the Stone-\v{C}ech compactification $\beta S$ of $S$ so that
$(\beta S,+)$ is a compact right topological semigroup (meaning that
for any $p\in\beta S$, the function $\rho_{p}:\beta S\rightarrow\beta S$
defined by $\rho_{p}(q)=q+p$ is continuous) with $S$ contained in
its topological center (meaning that for any $x\in S$, the function
$\lambda_{x}:\beta S\rightarrow\beta S$ defined by $\lambda_{x}(q)=x+q$
is continuous). Given $p,q\in\beta S$ and $A\subseteq S$, $A\in p+q$
if and only if $\{x\in S:-x+A\in q\}\in p$, where $-x+A=\{y\in S:x+y\in A\}$.}{\large\par}

{\large{}The following technical definitions are required,}{\large\par}
\begin{defn}
{\large{}\label{Defn 2} Let $\left(S,\ast\right)$ be a partial semigroup.}{\large\par}

{\large{}(a) For $s\in S$, $\varphi\left(s\right)=\left\{ t\in S:\:s\ast t\text{ is defined}\right\} $.}{\large\par}

{\large{}(b) For $H\in\mathcal{P}_{f}\left(S\right),\:\sigma\left(H\right)=\bigcap_{s\in H}\varphi\left(s\right)$.}{\large\par}

{\large{}(c) $\left(S,\ast\right)$ is adequate iff $\sigma\left(H\right)\neq\emptyset$
for all $H\in\mathcal{P}_{f}\left(S\right)$.}{\large\par}

{\large{}(d) $\delta S=\bigcap_{x\in S}\overline{\varphi\left(x\right)}=\bigcap_{H\in\mathcal{P}_{f}\left(S\right)}\overline{\sigma\left(H\right)}$.}{\large\par}
\end{defn}

{\large{}In fact, for a semigroup $S$, $\delta S=\beta S$ and adequacy
is needed in partial semigroup to guarantee that $\delta S\neq\emptyset$.
$(\delta S,+)$ is a compact right topological semigroup \cite[Theorem 2.10]{key-3}.
In particular, the combinatorially and algebraically defined large
sets are sometimes different for partial semigroups. Although, they
coincide for semigroups. In partial semigroup, each algebraically
defined large sets have a combinatorial version which is defined as
using a notions precede by \v{c}.}{\large\par}

{\large{}For details and recent advances on partial semigroup, one
can read \cite{key-2,key-4}.}{\large\par}
\begin{defn}
{\large{}\label{Definition 3} Let $(S,\ast)$ be a partial semigroup
and let $f$ be a sequence in $S$. Then $f$ is $adequate$ if and
only if}\\
{\large{}1. For each $F\in\mathcal{P}_{f}\left(\mathbb{N}\right)$,
$\prod_{t\in F}f(t)$ is defined.}\\
{\large{}2. For each $F\in\mathcal{P}_{f}\left(S\right)$, there exist
$m\in\mathbb{N}$, such that $FP\left(\langle f(t)\rangle_{t=m}^{\infty}\right)\subseteq\sigma(F)$}\\
{\large{}Define for an adequate partial semigroup $(S,\ast)$,
\[
\mathcal{F}=\left\{ f:f\,is\,an\,adequate\,sequence\,in\,S\right\} 
\]
}{\large\par}
\end{defn}

{\large{}It has shown in \cite[Theorem 4.1]{key-4} that in countable
adequate partial semigroup $adequate\,sequences$ exists.}{\large\par}

{\large{}For an adequate partial semigroup we have these definitions
of $IP$ set both algebraically and combinatorially.}{\large\par}
\begin{defn}
{\large{}\label{Defn 4} \cite[Definition 5.3]{key-6} Let $\left(S,\ast\right)$
be an adequate partial semigroup and suppose $A\subseteq S$.}{\large\par}

{\large{}(a) $A$ is $IP$ if and only if there exists an idempotent
$p\in\delta S$ such that $A\in p$.}{\large\par}

{\large{}(b) $A$ is \v{c}-$IP$ (combinatorially $IP$) if and only
if there exists a sequence $\left\langle x_{n}\right\rangle _{n=1}^{\infty}$
in $S$ such that for all $F\in\mathcal{P}_{f}\left(\mathbb{N}\right)$,}\\
{\large{} $\prod_{n\in F}x_{n}$ is defined and $\prod_{n\in F}x_{n}\in A$.}{\large\par}
\end{defn}

{\large{}In \cite[Theorem 5.4]{key-6}, McLeod gave an example of
an \v{c}-$IP$ set which is not $IP$. A set is said to be $IP^{\star}$
if it intersects every $IP$ set, i.e, it belongs to every idempotent
$p\in\delta S$. }{\large\par}

\section{{\large{}proof of main theorem}}
\begin{lem}
{\large{}\label{Lemma 5} Let $(S,+)$ be countable adequate partial
semigroup then, 
\[
\bigcap_{m=1}^{\infty}\overline{FS\left(\left\langle x_{n}\right\rangle _{n=m}^{\infty}\right)}\cap\delta S
\]
is a semigroup.}{\large\par}
\end{lem}

\begin{proof}
{\large{}Choose an adequate sequence $\left\langle x_{n}\right\rangle _{n=1}^{\infty}$
by \cite[Theorem 4.1]{key-4}. As for each $F\in\mathcal{P}_{f}\left(S\right)$,
there exist $m\in\mathbb{N}$, such that $FS\left(\langle x_{t}\rangle_{t=m}^{\infty}\right)\subseteq\sigma(F)$,
we have $FS\left(\langle x_{t}\rangle_{t=n}^{\infty}\right)\bigcap\sigma(H)\neq\emptyset$
for all $H\in\mathcal{P}_{f}\left(S\right)$ and $n\in\mathbb{N}$.
Now take the following set,
\[
B=\left\{ FS\left(\langle x_{t}\rangle_{t=n}^{\infty}\right)\bigcap\sigma(H):n\in\mathbb{N},H\in\mathcal{P}_{f}\left(S\right)\right\} 
\]
Now claim is $B$ satiesfies finite intersection property. To see
this take any two $H,F\in\mathcal{P}_{f}\left(S\right)$ and $m,n\in\mathbb{N}$.
Then,
\begin{align*}
\left(FS\left(\langle x_{t}\rangle_{t=n}^{\infty}\right)\bigcap\sigma(H)\right)\bigcap\left(FS\left(\langle x_{t}\rangle_{t=m}^{\infty}\right)\bigcap\sigma(F)\right)\\
=FS\left(\langle x_{t}\rangle_{t=n}^{\infty}\right)\bigcap FS\left(\langle x_{t}\rangle_{t=m}^{\infty}\right)\bigcap\sigma(H\cup F)\neq\emptyset
\end{align*}
}{\large\par}

{\large{}The last line follows from \cite[Theorem 4.1]{key-4}.}{\large\par}

{\large{}Now as $B$ satiesfies finite intersection property, we have,}\\
{\large{} $\bigcap_{m=1}^{\infty}\overline{FS\left(\left\langle x_{n}\right\rangle _{n=m}^{\infty}\right)}\cap\delta S$
is a semigroup from \cite[Lemma 5.6]{key-6}.}{\large\par}
\end{proof}
\begin{fact}
{\large{}\label{Fact 6} As $\bigcap_{m=1}^{\infty}\overline{FS\left(\left\langle x_{n}\right\rangle _{n=m}^{\infty}\right)}\cap\delta S$
is a semigroup, we have an idempotent $p\in\bigcap_{m=1}^{\infty}\overline{FS\left(\left\langle x_{n}\right\rangle _{n=m}^{\infty}\right)}\cap\delta S$
by \cite[Theorem 2.5, page 40]{key-5}.}{\large\par}
\end{fact}

\begin{defn}
{\large{}\label{Definition 7} Let $\left\langle y_{n}\right\rangle _{n=1}^{\infty}$
be an $adequate\,sequence$ in an adequate commutative partial semigroup
$(S,+)$ and let $k,m\in\mathbb{N}$. Then $\left\langle x_{n}\right\rangle _{n=1}^{m}$
is a product subsystem of $FS\left(\left\langle y_{n}\right\rangle _{n=k}^{\infty}\right)$
if and only if there exists a sequence $\left\langle H_{n}\right\rangle _{n=1}^{m}$
in $\mathcal{P}_{f}\left(\mathbb{N}\right)$ such that}{\large\par}

{\large{}(a) $\min H_{1}\geq k$, }{\large\par}

{\large{}(b) $\max H_{n}<\min H_{n+1}$for each $n\in\left\{ 1,2,\ldots,m-1\right\} $
and}{\large\par}

{\large{}(c) $x_{n}=\sum_{t\in H_{n}}y_{t}$ for each $n\in\left\{ 1,2,\ldots,m\right\} $.}{\large\par}
\end{defn}

{\large{}Now we prove a result which will be useful to our work.}{\large\par}
\begin{thm}
{\large{}\label{Theorem 8} Let $\left(S_{i},+\right),\,i=1,2,\ldots,l$
be $l$ countable adequate commutative partial semigroups and $l\in\mathbb{N}\setminus\{1\}$
and let $\left\langle x_{i,n}\right\rangle _{n=1}^{\infty}$ are sequences
in $S_{i}$ for each $i\in\{1,2,\ldots,l\}$. Let $m,r\in\mathbb{N}$
and $\times_{i=1}^{l}S_{i}=\bigcup_{j=1}^{r}D_{j}$. Then there exists
$j\in\left\{ 1,2,\ldots,r\right\} $ and for each $i\in\left\{ 1,2,\ldots,l-1\right\} $,
there exists a product subsystem $\left\langle y_{i,n}\right\rangle _{n=1}^{m}$
of $FS\left(\left\langle x_{i,n}\right\rangle _{n=1}^{\infty}\right)$
and there exists a product subsystem $\left\langle y_{l,n}\right\rangle _{n=1}^{\infty}$
of $FS\left(\left\langle x_{l,n}\right\rangle _{n=1}^{\infty}\right)$
such that
\[
\left(\times_{i=1}^{l-1}FS\left(\left\langle y_{i,n}\right\rangle _{n=1}^{m}\right)\right)\times FS\left(\left\langle y_{l,n}\right\rangle _{n=1}^{\infty}\right)\subset D_{j}.
\]
}{\large\par}
\end{thm}

\begin{proof}
{\large{}For each $i\in\left\{ 1,2,\ldots,l\right\} $, from fact
\ref{Fact 6}, let $p_{i}$ be the idempotents with
\[
p_{i}\in\bigcap_{k=1}^{\infty}\overline{FS\left(\left\langle x_{i,n}\right\rangle _{n=k}^{\infty}\right)}
\]
 For $\left(y_{1},y_{2},\ldots,y_{i-1}\right)\in\times_{i=1}^{l-1}S_{i}$
and $j\in\left\{ 1,2,\ldots,r\right\} $, take,}\\
{\large{} $B_{l}\left(y_{1},y_{2},\ldots,y_{l-1},j\right)=\left\{ y\in S_{l}:\:\left(y_{1},y_{2},\ldots,y_{l-1},y\right)\in D_{j}\right\} $.}{\large\par}

{\large{}Now, given $t\in\left\{ 2,3,\ldots,l-1\right\} $, assume
that $B_{t+1}\left(y_{1},y_{2},\ldots,y_{t},j\right)$ has been defined
for each $\left(y_{1},y_{2},\ldots,y_{t}\right)\in\times_{i=1}^{t}S_{i}$
and each $j\in\left(1,2,\ldots,r\right)$. Given 
\[
\left(y_{1},y_{2},\ldots,y_{t-1}\right)\in\times_{i=1}^{t-1}S_{i}\text{ and \ensuremath{j\in\left\{ 1,2,\ldots,r\right\} }},
\]
 let
\[
B_{t}\left(y_{1},y_{2},\ldots,y_{t-1},j\right)=\:\left\{ y\in S_{t}:\:B_{t+1}\left(y_{1},y_{2},\ldots,y_{t-1},y,j\right)\in p_{t+1}\right\} .
\]
}{\large\par}

{\large{}Finally, given that $B_{2}\left(x,j\right)$ has been defined
for each $x\in S_{1}$ and each $j\in\left\{ 1,2,\ldots,r\right\} $,
let $B_{1}\left(j\right)=\left\{ x\in S_{1}:\:B_{2}\left(x,j\right)\in p_{2}\right\} $.}\\
{\large{} We show by downward induction on $t$ that for each $t\in\left\{ 2,3,\ldots,l\right\} $
and each $\left(y_{1},y_{2},\ldots,y_{t-1}\right)\in\times_{i=1}^{t-1}S_{i}$,
\[
S_{t}=\cup_{j=1}^{r}B_{t}\left(y_{1},y_{2},\ldots,y_{t-1},j\right)
\]
 This is trivially true for $t=l$. }\\
{\large{}Assume $t\in\left\{ 2,3,\ldots,l-1\right\} $ and the statement
is true for $t+1$. }\\
{\large{}Let, $\left(y_{1},y_{2},\ldots,y_{t-1}\right)\in\times_{i=1}^{t-1}S_{i}$.
Given $y\in S_{t}$, one has that,
\[
S_{t+1}=\bigcup_{j=1}^{r}B_{t+1}\left(y_{1},y_{2},\ldots,y_{t-1},y,j\right).
\]
}{\large\par}

{\large{}Now pick $j\in\left\{ 1,2,\ldots,r\right\} $ such that $B_{t+1}\left(y_{1},y_{2},\ldots,y_{t-1},y,j\right)\in p_{t+1}$.
Then $y\in B_{t}\left(y_{1},y_{2},\ldots,y_{t-1},j\right)$.}{\large\par}

{\large{}Since for each $x\in S_{1},S_{2}=\cup_{j=1}^{r}B_{2}\left(x,j\right)$,
one sees similarly that $S_{1}=\cup_{j=1}^{r}B_{1}\left(j\right)$.
Pick $j\in\left\{ 1,2,\ldots,r\right\} $ such that $B_{1}\left(j\right)\in p_{1}$.
Pick a product subsystem $\left\langle y_{1,n}\right\rangle _{n=1}^{\infty}$
of $FS\left(\left\langle x_{1,n}\right\rangle _{n=1}^{\infty}\right)$
such that $FS\left(\left\langle y_{1,n}\right\rangle _{n=1}^{\infty}\right)\subseteq B_{1}\left(j\right)$.
Let $C_{2}=\cap\left\{ B_{2}\left(a,j\right):\:a\in FS\left(\left\langle y_{1,n}\right\rangle _{n=1}^{m}\right)\right\} $.
Since $FS\left(\left\langle y_{1,n}\right\rangle _{n=1}^{m}\right)$
is finite, we have $C_{2}\in p_{2}$ so pick a product subsystem $\left\langle y_{2,n}\right\rangle _{n=1}^{\infty}$
of $FS\left(\left\langle x_{2,n}\right\rangle _{n=1}^{\infty}\right)$
such that $FS\left(\left\langle y_{2,n}\right\rangle _{n=1}^{\infty}\right)\subseteq C_{2}$.
Let $t\in\left\{ 2,3,\ldots,l-1\right\} $ and assume $\left\langle y_{t,n}\right\rangle _{n=1}^{\infty}$
has been chosen. Let 
\[
C_{t+1}=\bigcap\left\{ B_{t+1}\left(a_{1},a_{2},\ldots,a_{t},j\right):\:\left(a_{1},a_{2},\ldots,a_{t}\right)\in\times_{i=1}^{t}FS\left(\left\langle y_{i,n}\right\rangle _{n=1}^{m}\right)\right\} .
\]
}{\large\par}

{\large{}Then $C_{t+1}\in p_{t+1}$ so pick a product subsystem $\left\langle y_{t+1,n}\right\rangle _{n=1}^{\infty}$
of $FS\left(\left\langle x_{t+1,n}\right\rangle _{n=1}^{\infty}\right)$
such that $FS\left(\left\langle y_{t+1,n}\right\rangle _{n=1}^{\infty}\right)\subseteq C_{t+1}$.
Then, }{\large\par}

{\large{}
\[
\left(\times_{i=1}^{l-1}FS\left(\left\langle y_{i,n}\right\rangle _{n=1}^{m}\right)\right)\times FS\left(\left\langle y_{l,n}\right\rangle _{n=1}^{\infty}\right)\subset D_{j}.
\]
which proves the theorem.}{\large\par}
\end{proof}
{\large{}To show our main theorem, the following definition is needed,
which is a partial semigroup version of \cite[Definition 18.12 page 469]{key-5}.}{\large\par}
\begin{defn}
{\large{}\label{Definition 9} Let $\langle y_{n}\rangle_{n=1}^{\infty}$
be an $adequate\,sequence$ in an adequate commutative partial semigroup
$(S,+)$ and let $k,m\in\mathbb{N}$. Then $\langle x_{n}\rangle_{n=1}^{m}$
is a weak product subsystem of $\langle y_{n}\rangle_{n=k}^{\infty}$
if and only if there exists a sequence $\langle H_{n}\rangle_{n=1}^{m}$
in $\mathcal{P}_{f}(\mathbb{N})$ such that, }{\large\par}

{\large{}(a) $H_{n}\cap H_{n+1}=\emptyset$ for each $n\neq k$ in
$\{1,2,\ldots,m\}$ and }{\large\par}

{\large{}(b) $x_{n}=\underset{t\in H_{n}}{\sum}y_{t}$ for each $n\in\{1,2,\ldots,m\}$}{\large\par}
\end{defn}

{\large{}Recall that in a product subsystem one requires that $\max H_{n}<\min H_{n+1}$.}{\large\par}
\begin{lem}
{\large{}\label{Lemma 10} Let $l\in\mathbb{N}$ and for each $i\in\left\{ 1,2,\ldots,l\right\} $,
$S_{i}$ be countable adequate commutative partial semigroup and let
$\left\langle x_{i,n}\right\rangle _{n=1}^{\infty}$ be an $adequate\,sequence$
in $S_{i}$. Let,}{\large\par}

{\large{}
\[
\mathscr{L}=\left\{ p\in\delta\left(\times_{i=1}^{l}S_{i}\right):\text{ for each }A\in p\text{ and each }m,k\in\mathbb{N},\right.
\]
}{\large\par}

{\large{}
\[
\text{ there exists for each }i\in\left\{ 1,2,...,l\right\} \text{ a weak product subsystem }
\]
}{\large\par}

{\large{}
\[
\left.\text{\ensuremath{\left\langle x_{i,n}\right\rangle _{n=1}^{m}\text{ of }}\ensuremath{\ensuremath{\left\langle y_{i,n}\right\rangle _{n=k}^{\infty}}} such that }\times_{i=1}^{l}FS\left(\left\langle x_{i,n}\right\rangle _{n=1}^{m}\right)\subseteq A\right\} 
\]
}{\large\par}

{\large{}Then $\mathscr{L}$ is a compact subsemigroup of $\delta\left(\times_{i=1}^{l}S_{i}\right)$.}{\large\par}
\end{lem}

\begin{proof}
{\large{}Since any product subsystem is also a weak product subsystem,
we have, from the assumption and theorem \ref{Theorem 8}, $\mathcal{L}\neq\emptyset$
. Since $\mathscr{L}$ is defined as the set of ultrafilters all of
whose members satisfy a given property, $\mathscr{L}$ is closed,
hence compact. To see that $\mathscr{L}$ is a semigroup, let $p,q\in\mathscr{L}$,
let $A\in p+q$ and let $m,k\in\mathbb{N}$. Then $\left\{ \bar{a}\in\times_{i=1}^{l}S_{i}:\:-\bar{a}+A\in q\right\} \in p$
so choose each $i\in\left\{ 1,2,\ldots,l\right\} $ a weak product
subsystem $\left\langle x_{i,n}\right\rangle _{n=1}^{m}$ of $FS\left(\left\langle y_{i,n}\right\rangle _{n=k}^{\infty}\right)$
such that 
\[
\times_{i=1}^{l}FS\left(\left\langle x_{i,n}\right\rangle _{n=1}^{m}\right)\subseteq\left\{ \bar{a}\in\times_{i=1}^{l}S_{i}:\:-\bar{a}+A\in q\right\} .
\]
 Given $i\in\left\{ 1,2,\ldots,l\right\} $ and $n\in\left\{ 1,2,\ldots,m\right\} $,
pick $H_{i,n}\in\mathcal{P}_{f}\left(\mathbb{N}\right)$ with min
$H_{i,n}\geq k$ such that,}\\
{\large{} $x_{i,n}=\sum_{t\in H_{i,n}}y_{i,t}$ and if $1\leq n<s\leq m$,
then $H_{i,n}\cap H_{i,s}=\emptyset$. Let $r=\text{max }\left(\bigcup_{i=1}^{l}\bigcup_{n=1}^{m}H_{i,m}\right)+1$
and let 
\[
B=\bigcap\left\{ -\bar{a}+A:\:\bar{a}\in\times_{i=1}^{l}FS\left(\left\langle x_{i,n}\right\rangle _{n=1}^{m}\right)\right\} .
\]
 Then $B\in q$ so choose for each $i\in\left\{ 1,2,\ldots,l\right\} $
a weak product subsystem $\left\langle z_{i,n}\right\rangle _{n=1}^{m}$
of $FS\left(\left\langle y_{i,n}\right\rangle _{n=r}^{\infty}\right)$
such that $\times_{i=1}^{l}FS\left(\left\langle z_{i,n}\right\rangle _{n=1}^{m}\right)\subseteq B$.
}\\
{\large{}Given $i\in\left\{ 1,2,\ldots,l\right\} $ and $n\in\left\{ 1,2,\ldots,m\right\} $,
pick $K_{i,n}\in\mathcal{P}_{f}\left(\mathbb{N}\right)$ with min
$K_{i,n}\geq r$ such that $z_{i,n}=\sum_{t\in K_{i,n}}y_{i,t}$ and
if $1\leq n<s\leq m$, then $K_{i,n}\cap K_{i,s}=\emptyset$. For
$i\in\left\{ 1,2,\ldots,l\right\} $ and $n\in\left\{ 1,2,\ldots,m\right\} $,
let $L_{i,n}=H_{i,n}\cup K_{i,n}$. Then 
\[
\sum_{t\in L_{i,n}}y_{i,t}=\sum_{t\in H_{i,n}}y_{i,t}+\sum_{t\in K_{i,n}}y_{i,t}=x_{i,t}+z_{i,t}
\]
 and if $1\leq n<s\leq m$, then $L_{i,n}\cap L_{i,s}=\emptyset$
. Thus for each $i\in\left\{ 1,2,\ldots,l\right\} $, $\left\langle x_{i,n}+z_{i,n}\right\rangle _{n=1}^{m}$
is a weak product subsystem of $FS\left(\left\langle y_{i,n}\right\rangle _{n=k}^{\infty}\right)$.}{\large\par}

{\large{}Finally we claim that $\times_{i=1}^{l}FS\left(\left\langle x_{i,n}+z_{i,n}\right\rangle _{n=1}^{m}\right)\subseteq A$.
To this end let $\bar{c}\in\times_{i=1}^{l}FS\left(\left\langle x_{i,n}+z_{i,n}\right\rangle _{n=1}^{m}\right)$
be given. For each $i\in\left\{ 1,2,\ldots,l\right\} $, pick $F_{i}\subseteq\left\{ 1,2,\ldots,m\right\} $
such that $c_{i}=\sum_{t\in F_{i}}\left(x_{i,t}+z_{i,t}\right)$ and
let $a_{i}=\sum_{t\in F_{i}}x_{i,t}$ and $b_{i}=\sum_{t\in F_{i}}z_{i,t}$.
Then $\bar{b}\in\times_{i=1}^{l}FS\left(\left\langle z_{i,n}\right\rangle _{n=1}^{m}\right)$
so $\bar{b}\in B$. }\\
{\large{}Since $\bar{a}\in\times_{i=1}^{l}FS\left(\left\langle x_{i,n}\right\rangle _{n=1}^{m}\right)$
one has that $\bar{b}\in-\bar{a}+A$ so that $\bar{a}+\bar{b}\in A$.
Since each $S_{i}$ is commutative, we have for each $i\in\left\{ 1,2,\ldots,l\right\} $
that 
\[
\bar{c}=\sum_{t\in F_{i}}\left(x_{i,t}+z_{i,t}\right)=\sum_{t\in F_{i}}x_{i,t}+\sum_{t\in F_{i}}z_{i,t}=\bar{a}.\bar{b}
\]
 as required. }{\large\par}
\end{proof}
{\large{}Now we will prove our main theorem.}{\large\par}
\begin{thm}
{\large{}\label{Theorem 11} Let $l\in\mathbb{N}$ and for each $i\in\left\{ 1,2,\ldots,l\right\} $,
let $S_{i}$ be countable adequate commutative partial semigroup and
let $\left\langle y_{i,n}\right\rangle _{n=1}^{\infty}$ be a sequence
in $S_{i}$. Let $C$ be an $IP^{*}$ set in $\times_{i=1}^{l}S_{i}$,
and let $m\in\mathbb{N}$. Then for each $i\in\left\{ 1,2,\ldots,l\right\} $
there is a weak product subsystem $\left\langle x_{i,n}\right\rangle _{n=1}^{m}$
of $FS\left(\left\langle y_{i,n}\right\rangle _{n=1}^{\infty}\right)$
such that $\times_{i=1}^{l}FS\left(\left\langle x_{i,n}\right\rangle _{n=1}^{m}\right)\subseteq C$. }{\large\par}
\end{thm}

\begin{proof}
{\large{}Let $\mathscr{L}$ be as in Lemma \ref{Lemma 10}. Then $\mathscr{L}$
is a compact subsemigroup of $\delta\left(\times_{i=1}^{l}S_{i}\right)$.
So there is an idempotent $p\in\mathscr{L}$. Let $C$ be an be an
$IP^{*}$ set in $\times_{i=1}^{l}S_{i}$. Hence $C\in p$. Thus by
the definition of $\mathscr{L}$ for each $i\in\left\{ 1,2,\ldots,l\right\} $,
there is a weak product subsystem $\left\langle x_{i,n}\right\rangle _{n=1}^{m}$
of $FS\left(\left\langle y_{i,n}\right\rangle _{n=1}^{\infty}\right)$
such that $\times_{i=1}^{l}FP\left(\left\langle x_{i,n}\right\rangle _{n=1}^{m}\right)\subseteq C$.}{\large\par}

{\large{}This completes the proof.}{\large\par}
\end{proof}
\textbf{\large{}Acknowledgment: }{\large{}The author thanks to Dibyendu
De for his constant inspiration to do this project.}{\large\par}

\end{document}